\documentclass[11pt]{article}
\usepackage[a4paper,margin=30mm]{geometry}
\usepackage{amsmath,amssymb,amsthm,mathtools}
\usepackage{array}
\usepackage{booktabs}
\usepackage[hidelinks]{hyperref}
\usepackage{microtype}

\newtheorem{theorem}{Theorem}[section]
\newtheorem{proposition}[theorem]{Proposition}
\newtheorem{lemma}[theorem]{Lemma}
\newtheorem{corollary}[theorem]{Corollary}
\theoremstyle{definition}
\newtheorem{definition}[theorem]{Definition}
\newtheorem{question}[theorem]{Question}
\theoremstyle{remark}
\newtheorem{remark}[theorem]{Remark}

\hypersetup{
    pdftitle={Large Sidon Subsets and Pair-Sum Multiplicities of Distinct Multinomial Coefficients},
    pdfauthor={Felix Huber},
    pdfsubject={Sidon subsets and additive pair-sum collisions of distinct multinomial coefficient values},
    pdfkeywords={Sidon set, multinomial coefficient, pair sum, additive collision, representation multiplicity, additive energy, integer partition}
}

\title{Large Sidon Subsets and Pair-Sum Multiplicities\\of Distinct Multinomial Coefficients}
\author{Felix Huber\\[-1mm]
\small Independent Researcher\\[-1mm]
\small Maifhofhalde 20, 6006 Luzern, Switzerland\\[-1mm]
\small \texttt{felix.68@gmx.ch}\\[-1mm]
\small ORCID: \href{https://orcid.org/0009-0005-1568-1579}{\texttt{0009-0005-1568-1579}}}
\date{}

\begin{document}
\maketitle
\begin{abstract}
For a positive integer $n$, let $\mathcal{M}_n$ be the set of distinct multinomial coefficient values $n!/(p_1!\cdots p_t!)$, where $(p_1,\ldots,p_t)$ ranges over the integer partitions of $n$. We study two complementary aspects of the additive structure of $\mathcal{M}_n$: the maximum cardinality $s(n)$ of a Sidon subset and the multiplicities of unordered pair sums in the full set. A product embedding, strongly Sidon subsets, and estimates for prime partitions give
\[
\liminf_{n\to\infty}
\frac{\log(s(n)/n)\log\log n}{\sqrt{\log n}}
\geq\frac{\pi}{\sqrt3}.
\]
Long arithmetic progressions and separated copies of them yield the complementary Sidon-defect bound
\[
\liminf_{n\to\infty}
\frac{|\mathcal{M}_n|-s(n)}{n^{3/2}\sqrt{\log n}}
\geq\frac{2}{3\sqrt3}.
\]
Writing $r_n(t)$ for the number of unordered representations $t=x+y$ with $x,y\in\mathcal{M}_n$, we study the collision excess $C(n)$, the number $D(n)$ of multiply represented sums, the maximum multiplicity $\mu(n)$, and the cumulative profile $T(n,k)$. The same separated progressions give
\[
\liminf_{n\to\infty}\frac{C(n)}{n^2\log n}\geq\frac18,
\qquad
\liminf_{n\to\infty}\frac{D(n)}{n^{3/2}\sqrt{\log n}}
\geq\frac{4}{3\sqrt3},
\]
and
\[
\liminf_{n\to\infty}\frac{\mu(n)}{\sqrt{n\log n}}\geq\frac12.
\]
We further obtain a scaled lower envelope for the full multiplicity profile, an additive-energy bound, stabilization results with respect to the number of variables, and an explicit family of trinomial collisions. Exact values of $s(n)$ through $n=16$ and pair-sum statistics through $n=20$ are reported together with reproducible computational material.
\end{abstract}

\section{Introduction}\label{sec:intro}
Throughout the paper, $\log$ denotes the natural logarithm. A finite set $S$ of integers is a \emph{Sidon set} if every equality
\[
x+y=u+v,\qquad x,y,u,v\in S,
\]
is trivial, meaning that the unordered pairs $\{x,y\}$ and $\{u,v\}$ coincide. Equivalently, all unordered pair sums, with equal summands allowed, are distinct. Sidon sets are classical objects of additive combinatorics; see O'Bryant's survey~\cite{OBryant}.

For each $n\geq1$, let $\mathcal{M}_n$ denote the set of distinct multinomial coefficient values associated with the integer partitions of $n$; the formal definition is given in Section~\ref{sec:multinomial-sets}. Its cardinality was studied by Andrews, Knopfmacher, and Zimmermann~\cite{AKZ}. We consider
\[
s(n)=\max\{|S|:S\subseteq\mathcal{M}_n,\ S\text{ is Sidon}\}.
\]
The first exact values are
\[
1, 2, 3, 5, 6, 8, 11, 15, 20, 25, 30, 39, 45, 55, 66, 83
\]
for $1\leq n\leq16$.

A complementary viewpoint retains the full set $\mathcal{M}_n$ and measures its additive collisions. For an integer $t$, put
\[
r_n(t)=\#\{(x,y):x,y\in\mathcal{M}_n,\ x\leq y,\ x+y=t\}.
\]
Representations are therefore counted by unordered pairs, with equal summands allowed. Define
\begin{align*}
C(n)&=\sum_t\max(r_n(t)-1, 0),&
D(n)&=\#\{t:r_n(t)\geq2\},\\
\mu(n)&=\max_t r_n(t),&
T(n,k)&=\#\{t:r_n(t)\geq k\}\qquad(k\geq2).
\end{align*}
Pair-sum representation functions are standard objects in additive number theory~\cite{Nathanson}. The quantities $s(n)$ and $C(n),D(n),\mu(n),T(n,k)$ describe opposite but closely related aspects of the same additive structure: the former asks how much of $\mathcal{M}_n$ can be retained while eliminating all collisions, whereas the latter measure the collisions present in the complete set.

The common structural tool is a product embedding for multinomial coefficient sets. It supports two amplification mechanisms. Strongly Sidon subsets can be placed at sufficiently separated multiplicative scales, producing large Sidon subsets. Long arithmetic progressions can likewise be replicated at separated scales. Their element intervals force independent omissions from a Sidon subset, while their pair-sum intervals support disjoint contributions to the complete representation profile.

The first main result concerns large Sidon subsets.
\begin{theorem}\label{thm:main-lower}
We have
\[
\liminf_{n\to\infty}
\frac{\log(s(n)/n)\log\log n}{\sqrt{\log n}}
\geq\frac{\pi}{\sqrt3}.
\]
Equivalently, for every $\varepsilon>0$,
\[
s(n)\geq n\exp\left(\left(\frac{\pi}{\sqrt3}-\varepsilon\right)
\frac{\sqrt{\log n}}{\log\log n}\right)
\]
for all sufficiently large $n$. In particular, $s(n)=\Omega(n\log n)$, and $s(n)/n$ grows faster than every fixed power of $\log n$.
\end{theorem}

The arithmetic-progression construction also gives an obstruction in the opposite direction.
\begin{theorem}\label{thm:defect}
With $M(n)$ as in~\eqref{eq:M-cardinality},
\[
\liminf_{n\to\infty}
\frac{M(n)-s(n)}{n^{3/2}\sqrt{\log n}}
\geq\frac{2}{3\sqrt3}.
\]
\end{theorem}

The same progressions provide detailed lower bounds for the pair-sum statistics. With $N(q)$ as in~\eqref{eq:Nq-definition}, the finite result is the following.
\begin{theorem}\label{thm:finite-main}
Let $q\geq3$ and $n\geq N(q)+q$, and put
\[
h=n-N(q)-q+1.
\]
Then
\begin{align*}
C(n)&\geq h\frac{(q-1)(q-2)}2,\\
D(n)&\geq h(2q-5),\\
T(n,k)&\geq h(2q-4k+3)
\end{align*}
for every $2\leq k\leq\lceil q/2\rceil$. Moreover,
\[
\mu(n)\geq\left\lceil\frac q2\right\rceil.
\]
\end{theorem}

Optimizing $q$ gives three asymptotic consequences.
\begin{theorem}\label{thm:asymptotic-main}
We have
\[
\liminf_{n\to\infty}\frac{C(n)}{n^2\log n}\geq\frac18,
\qquad
\liminf_{n\to\infty}\frac{D(n)}{n^{3/2}\sqrt{\log n}}
\geq\frac{4}{3\sqrt3},
\]
and
\[
\liminf_{n\to\infty}\frac{\mu(n)}{\sqrt{n\log n}}\geq\frac12.
\]
\end{theorem}

The complete profile has a nontrivial lower envelope. For $0\leq\alpha<1/2$, define
\[
c_\alpha=\frac{2\alpha+\sqrt{4\alpha^2+3}}3,
\qquad
\Phi(\alpha)=2(c_\alpha-2\alpha)(1-c_\alpha^2).
\]
\begin{theorem}\label{thm:profile-limit}
Let $k_n\geq2$ be integers such that
\[
\frac{k_n}{\sqrt{n\log n}}\longrightarrow\alpha,
\qquad 0\leq\alpha<\frac12.
\]
Then
\[
\liminf_{n\to\infty}
\frac{T(n,k_n)}{n^{3/2}\sqrt{\log n}}
\geq\Phi(\alpha).
\]
In particular, $\Phi(0)=4/(3\sqrt3)$ and $\Phi(\alpha)\to0$ as $\alpha\uparrow1/2$.
\end{theorem}

Section~\ref{sec:multinomial-sets} establishes the common product embedding. Sections~\ref{sec:sidon-lifting}--\ref{sec:lower} develop the Sidon constructions and prove Theorem~\ref{thm:main-lower}. Section~\ref{sec:progressions} gives the shared arithmetic-progression construction, which is applied to the Sidon defect in Section~\ref{sec:defect}. The second main part begins in Section~\ref{sec:collision}; Sections~\ref{sec:profile}--\ref{sec:asymmult} establish the finite and asymptotic multiplicity results. Sections~\ref{sec:variables} and~\ref{sec:trinomial} treat the number of variables and explicit trinomial identities. Section~\ref{sec:computations} collects computations, OEIS connections, and open problems.
\section{Multinomial coefficient sets and product embeddings}
\label{sec:multinomial-sets}

For $n\geq1$, define
\[
\mathcal{M}_n
=
\left\{
\frac{n!}{p_1!\cdots p_t!}:
 p_1+\cdots+p_t=n,\quad
 p_1\geq\cdots\geq p_t\geq1
\right\}.
\]
Thus $\mathcal{M}_n$ is the set of distinct multinomial coefficient values with upper entry $n$. We write
\begin{equation}\label{eq:M-cardinality}
M(n)=|\mathcal{M}_n|.
\end{equation}
The function $M(n)$ was studied by Andrews, Knopfmacher, and Zimmermann~\cite{AKZ}. In particular, they proved that $M(n)=o(p(n))$, where $p(n)$ denotes the ordinary partition function.

For finite sets $A_1,\ldots,A_\ell$ of positive integers, write
\[
A_1\cdots A_\ell
=
\{a_1\cdots a_\ell:a_i\in A_i\}.
\]
The basic structural property used throughout the paper is the following product embedding.

\begin{proposition}[Product embedding]
\label{prop:product-embedding}
Let $n_1,\ldots,n_\ell$ be positive integers satisfying
\[
n_1+\cdots+n_\ell=n.
\]
Then
\[
\binom{n}{n_1,\ldots,n_\ell}
\mathcal{M}_{n_1}\cdots\mathcal{M}_{n_\ell}
\subseteq
\mathcal{M}_n.
\]
\end{proposition}

\begin{proof}
For each $i$, choose a partition
\[
p_{i,1}+\cdots+p_{i,t_i}=n_i
\]
representing an element
\[
x_i
=
\frac{n_i!}{p_{i,1}!\cdots p_{i,t_i}!}
\in\mathcal{M}_{n_i}.
\]
Concatenating these partitions gives a partition of $n$. The corresponding multinomial coefficient is
\[
\frac{n!}
{\displaystyle\prod_{i=1}^{\ell}\prod_{j=1}^{t_i}p_{i,j}!}
=
\binom{n}{n_1,\ldots,n_\ell}
x_1\cdots x_\ell,
\]
which proves the inclusion.
\end{proof}

In particular, if $r+m=n$, then
\[
\binom{n}{r}\mathcal{M}_r\mathcal{M}_m
\subseteq
\mathcal{M}_n.
\]
Since $1\in\mathcal{M}_m$, this also gives
\[
\binom{n}{r}\mathcal{M}_r
\subseteq
\mathcal{M}_n.
\]

The special case $m=1$ will be used repeatedly.

\begin{corollary}
\label{cor:successive-embedding}
For every $n\geq1$,
\[
(n+1)\mathcal{M}_n
\subseteq
\mathcal{M}_{n+1}.
\]
\end{corollary}

\begin{proof}
Apply Proposition~\ref{prop:product-embedding} with $n_1=n$ and $n_2=1$, noting that $\mathcal{M}_1=\{1\}$.
\end{proof}

Consequently, every additive relation among elements of $\mathcal{M}_n$ may be lifted to $\mathcal{M}_{n+1}$ by multiplication by $n+1$. More generally, the product embedding allows additive configurations in lower-order multinomial coefficient sets to be transferred, combined, and separated inside sets of higher order. The two principal applications developed below are the construction of large Sidon subsets and the construction of many disjoint families of pair-sum collisions.
\section{Sidon subsets and elementary lifting}\label{sec:sidon-lifting}
The elementary product embedding already gives useful explicit lower bounds for $s(n)$.
\begin{lemma}\label{lem:adjoin-one}
Let $S$ be a Sidon set of positive integers, and let $q\geq3$. Then
\[
\{1\}\cup qS
\]
is Sidon.
\end{lemma}

\begin{proof}
The pair sums are of the three forms
\[
1+1,\qquad 1+qx,\qquad qx+qy.
\]
Their residues modulo $q$ are respectively $2, 1, 0$, so sums of different forms cannot coincide. Equalities within the second form are trivial, while an equality within the third form reduces, after division by $q$, to an equality of pair sums in $S$.
\end{proof}

\begin{proposition}\label{prop:elementary-lifting}
For $2\leq r\leq n$,
\[
s(n)\geq s(r)+n-r.
\]
Consequently, $s(n)-n$ is nondecreasing for $n\geq2$.
\end{proposition}

\begin{proof}
If $S\subseteq\mathcal{M}_k$ is Sidon, Proposition~\ref{prop:product-embedding} and Lemma~\ref{lem:adjoin-one} show that
\[
\{1\}\cup(k+1)S\subseteq\mathcal{M}_{k+1}
\]
is Sidon. Iterating from $k=r$ to $k=n-1$ proves the result.
\end{proof}

The iterated construction is explicit.

\begin{proposition}\label{prop:explicit-lifting}
Let $2\leq r\leq n$, and let $S\subseteq\mathcal{M}_r$ be Sidon. Then
\[
L_{n,r}(S)=\left\{1,n,n(n-1),\ldots,\frac{n!}{(r+1)!}\right\}\cup\frac{n!}{r!}S
\]
is a Sidon subset of $\mathcal{M}_n$, of cardinality $|S|+n-r$.
\end{proposition}

\begin{proof}
This is the result of iterating $U\mapsto\{1\}\cup(k+1)U$ for $k=r,\ldots,n-1$.
\end{proof}

A simple special case gives a superincreasing family.

\begin{corollary}\label{cor:factorial-family}
For every $n\geq1$,
\[
\{1,n,n(n-1),\ldots,n!\}
\]
is a Sidon subset of $\mathcal{M}_n$. In particular, $s(n)\geq n$.
\end{corollary}

\begin{proof}
Write
\[
f_k=\frac{n!}{(n-k)!},\qquad 0\leq k\leq n-1.
\]
Each $f_k$ belongs to $\mathcal{M}_n$, corresponding to the partition $(n-k,1,\ldots,1)$. Moreover,
\[
\frac{f_k}{f_{k-1}}=n-k+1\geq2.
\]
Suppose that $f_i+f_j=f_u+f_v$, where $i\leq j$, $u\leq v$, and, without loss of generality, $j\leq v$. If $j<v$, then
\[
f_i+f_j\leq2f_{v-1}\leq f_v<f_u+f_v,
\]
a contradiction. Hence $j=v$, and cancellation gives $f_i=f_u$, so $i=u$.
\end{proof}

The exact value $s(16)=83$, proved computationally in Section~\ref{sec:computations}, gives the concrete global estimate
\[
s(n)\geq n+67,\qquad n\geq16.
\]
\section{Strongly Sidon sets and separated dilations}\label{sec:strong}
The elementary lifting above adds one element at each step. To obtain a superlinear lower bound, we use a strengthened uniqueness condition.

\begin{definition}
A finite set $S$ of positive integers is \emph{strongly Sidon} if all sums containing zero, one, or two elements of $S$ are distinct, with repetition allowed in a two-element sum.

Equivalently, $S$ is strongly Sidon if $\{0\}\cup S$ is Sidon. Let
\[
b(n)=\max\{|S|:S\subseteq\mathcal{M}_n,\ S\text{ is strongly Sidon}\}.
\]
Clearly $s(n)\geq b(n)$.
\end{definition}

\begin{lemma}[Separated-dilation lemma]\label{lem:separated}
Let $S$ be strongly Sidon, let $L=\max S$, and let
\[
1\leq y_1<\cdots<y_h
\]
satisfy
\[
y_{i+1}>2Ly_i,\qquad 1\leq i<h.
\]
Then the sets $y_iS$ are pairwise disjoint, and
\[
U=\bigcup_{i=1}^h y_iS
\]
is strongly Sidon. In particular, $|U|=h|S|$.
\end{lemma}

\begin{proof}
Suppose that two sums, each containing zero, one, or two elements of $U$, are equal. Let $j$ be the largest index of a scale occurring in either sum. The contribution at scale $y_j$ is $y_jd$, where
\[
d\in D:=\{0\}\cup S\cup(S+S).
\]
Strong Sidonicity says that every element of $D$ has a unique representation using zero, one, or two elements of $S$.

If $j=1$, there is no contribution from a smaller scale. If $j>1$, the contribution from all smaller scales on either side is between $0$ and $2Ly_{j-1}$, so the difference between the two smaller-scale contributions has absolute value less than $y_j$. If the two coefficients $d$ at scale $y_j$ were different, their contributions would differ by at least $y_j$, which is impossible. Thus the coefficients are equal. Strong Sidonicity identifies the same multiset of zero, one, or two elements of $S$ at scale $y_j$. Cancelling these terms and descending through the remaining scales proves uniqueness. In particular, one-element representations are unique, so the sets $y_iS$ are pairwise disjoint.
\end{proof}

\begin{proposition}[Strong-basis recursion]\label{prop:strong-recursion}
Let $1\leq r<n$, and let $S\subseteq\mathcal{M}_r$ be strongly Sidon with $L=\max S$. If $n\geq r+2L$, then
\[
b(n)\geq |S|(n-r-2L+1).
\]
\end{proposition}

\begin{proof}
Put $m=n-r$ and define
\[
y_j=\frac{m!}{(m-j)!},\qquad 0\leq j\leq m-2L.
\]
These values belong to $\mathcal{M}_m$. For $0\leq j<m-2L$,
\[
\frac{y_{j+1}}{y_j}=m-j\geq2L+1>2L.
\]
Lemma~\ref{lem:separated} shows that $\bigcup_j y_jS$ is strongly Sidon and has cardinality $|S|(m-2L+1)$. Proposition~\ref{prop:product-embedding} places its common dilate
\[
\binom{n}{r}\left(\bigcup_j y_jS\right)
\]
inside $\mathcal{M}_n$.
\end{proof}

\begin{corollary}\label{cor:strong-recursion-crude}
For integers $1\leq r<n$ with $n\geq r+2r!$,
\[
b(n)\geq b(r)(n-r-2r!+1).
\]
\end{corollary}

\begin{proof}
Apply Proposition~\ref{prop:strong-recursion} to a largest strongly Sidon subset of $\mathcal{M}_r$, using the crude bound $\max S\leq r!$.
\end{proof}

\section{The large-subset theorem}\label{sec:lower}
We first record an entirely explicit consequence of the separated-dilation construction, and then combine Proposition~\ref{prop:strong-recursion} with general results on Sidon subsets and prime partitions.

\subsection{An explicit elementary bound}
\begin{lemma}\label{lem:Tr}
For every integer $r\geq3$, the set
\[
T_r=\left\{\frac{r!}{(r-k)!}:0\leq k\leq r-2\right\}
=\left\{1,r,r(r-1),\ldots,\frac{r!}{2}\right\}
\]
is strongly Sidon. In particular, $|T_r|=r-1$ and $\max T_r=r!/2$.
\end{lemma}

\begin{proof}
Write
\[
t_k=\frac{r!}{(r-k)!},\qquad 0\leq k\leq r-2.
\]
Then
\[
\frac{t_k}{t_{k-1}}=r-k+1\geq3
\]
for $1\leq k\leq r-2$. Hence
\[
\sum_{j=0}^{k-1}t_j\leq t_k\sum_{h=1}^k3^{-h}<\frac{t_k}{2},
\]
so
\[
t_k>2\sum_{j<k}t_j.
\]
Suppose that two sums containing zero, one, or two elements of $T_r$ are equal. At the largest index $k$ where their multiplicities differ, the contribution has absolute value at least $t_k$, whereas the total possible contribution of all smaller terms is at most $2\sum_{j<k}t_j<t_k$, a contradiction.
\end{proof}

\begin{proposition}\label{prop:elementary-strong}
For every $r\geq3$ and every $n\geq r+r!$,
\[
s(n)\geq b(n)\geq(r-1)(n-r-r!+1).
\]
\end{proposition}

\begin{proof}
Apply the separated-dilation construction with $T_r$. Since $\max T_r=r!/2$, Proposition~\ref{prop:strong-recursion} gives
\[
b(n)\geq |T_r|(n-r-2\max T_r+1)=(r-1)(n-r-r!+1).
\]
\end{proof}

\begin{corollary}\label{cor:elementary-asymptotic}
As $n\to\infty$,
\[
s(n)\geq(1-o(1))\frac{n\log n}{\log\log n}.
\]
\end{corollary}

\begin{proof}
Let $r$ be the largest integer satisfying $r!\leq n/\log n$. Stirling's formula and maximality give
\[
r\log r\sim\log n,\qquad r\sim\frac{\log n}{\log\log n}.
\]
Moreover, $r/n\to0$ and $r!/n\leq1/\log n\to0$. Proposition~\ref{prop:elementary-strong} now gives
\[
s(n)\geq(r-1)(n-r-r!+1)=(1-o(1))\frac{n\log n}{\log\log n}.
\]
\end{proof}

\subsection{A stronger bound from prime partitions}
\begin{lemma}[Strong-Sidon extraction]\label{lem:strong-extraction}
Every finite Sidon set $A$ of positive integers contains a strongly Sidon subset of cardinality at least $|A|/3$.
\end{lemma}

\begin{proof}
Let $I=(1/3, 2/3)\subset\mathbb{R}/\mathbb{Z}$. For $\theta\in[0, 1)$, put
\[
T_\theta=\{a\in A:\{\theta a\}\in I\},
\]
where $\{x\}$ denotes the fractional part of $x$. For each nonzero integer $a$, the map $\theta\mapsto\theta a\pmod 1$ preserves Lebesgue measure. Hence
\[
\int_0^1|T_\theta|\,d\theta=\frac{|A|}{3},
\]
so some $T_\theta$ has at least $|A|/3$ elements.

The interval $I$ is sum-free modulo $1$: if $\alpha,\beta\in I$, then $\alpha+\beta\pmod1\notin I$. Consequently no $x,y,z\in T_\theta$ satisfy $x+y=z$, including the case $x=y$. Since $T_\theta\subseteq A$ remains Sidon, and for positive Sidon sets the only additional obstruction to strong Sidonicity is a relation $x+y=z$, the set $T_\theta$ is strongly Sidon.
\end{proof}

\begin{corollary}\label{cor:b-vs-s}
For every $n\geq1$,
\[
\left\lceil\frac{s(n)}{3}\right\rceil\leq b(n)\leq s(n).
\]
Consequently, for $1\leq r<n$ and $n\geq r+2r!$,
\[
s(n)\geq\left\lceil\frac{s(r)}{3}\right\rceil(n-r-2r!+1).
\]
\end{corollary}

\begin{proof}
Apply Lemma~\ref{lem:strong-extraction} to a largest Sidon subset of $\mathcal{M}_n$ to obtain the first inequality; the second is immediate from the definition of $b(n)$. The recursive estimate follows from $s(n)\geq b(n)$ and Corollary~\ref{cor:strong-recursion-crude}.
\end{proof}

Let $p_{\mathbb{P}}(r)$ denote the number of unrestricted partitions of $r$ into prime parts.

\begin{proposition}\label{prop:b-asymptotic}
We have
\[
\liminf_{r\to\infty}\frac{\log b(r)}{\sqrt{r/\log r}}\geq\frac{\pi}{\sqrt3}.
\]
\end{proposition}

\begin{proof}
Let $H(m)$ be the minimum, over all $m$-element subsets $B\subseteq\mathbb{Z}$, of the maximum cardinality of a Sidon subset of $B$. Abbott proved that there is an absolute constant $c_0>0$ such that
\[
H(m)\geq c_0\sqrt m
\]
for all sufficiently large $m$~\cite{Abbott}. More recently, Bailleul and Riblet proved the sharper estimate
\[
H(m)\geq\left(\frac{1}{3\sqrt3}+o(1)\right)\sqrt m
\]
as $m\to\infty$~\cite{BailleulRiblet}. The argument below requires only the established order of magnitude $H(m)\gg\sqrt m$. Since $M(r)\geq r$ by Corollary~\ref{cor:factorial-family}, the set $\mathcal{M}_r$ therefore contains a Sidon subset $A$ with
\[
|A|\geq c_0\sqrt{M(r)}.
\]
Lemma~\ref{lem:strong-extraction} gives
\[
b(r)\geq\frac{c_0}{3}\sqrt{M(r)}.
\]
By Theorem~1 of Andrews, Knopfmacher, and Zimmermann~\cite{AKZ},
\[
M(r)\geq p_{\mathbb{P}}(r)
\]
for every positive integer $r$. The classical asymptotic for unrestricted prime partitions is
\[
\log p_{\mathbb{P}}(r)\sim\frac{2\pi}{\sqrt3}\sqrt{\frac{r}{\log r}}
\]
\cite{Kerawala,Vaughan}. Hence
\[
\log b(r)\geq\frac12\log p_{\mathbb{P}}(r)+O(1)
=\left(\frac{\pi}{\sqrt3}+o(1)\right)\sqrt{\frac{r}{\log r}},
\]
which proves the proposition.
\end{proof}

\begin{proof}[Proof of Theorem~\ref{thm:main-lower}]
Choose $r=r(n)$ maximal with $4r!\leq n$. For all sufficiently large $n$, the condition $n\geq r+2r!$ required in Corollary~\ref{cor:strong-recursion-crude} is then satisfied. Stirling's formula gives
\[
r\sim\frac{\log n}{\log\log n},\qquad \log r\sim\log\log n,
\]
and hence
\[
\sqrt{\frac{r}{\log r}}\sim\frac{\sqrt{\log n}}{\log\log n}.
\]
Moreover,
\[
n-r-2r!+1\geq\frac n2-r+1=\left(\frac12-o(1)\right)n.
\]
Corollary~\ref{cor:strong-recursion-crude} therefore yields
\[
\frac{s(n)}{n}\geq\left(\frac12-o(1)\right)b(r),
\]
and hence
\[
\log\frac{s(n)}{n}\geq\log b(r)+O(1).
\]
Combining this with Proposition~\ref{prop:b-asymptotic} and the preceding asymptotics gives
\[
\liminf_{n\to\infty}\frac{\log(s(n)/n)\log\log n}{\sqrt{\log n}}\geq\frac{\pi}{\sqrt3}.
\]
The equivalent $\varepsilon$-form follows directly. Since $\sqrt{\log n}/\log\log n$ grows faster than $\log\log n$, the lower bound also implies $s(n)=\Omega(n\log n)$; indeed, $s(n)/n$ grows faster than every fixed power of $\log n$.
\end{proof}

\section{Arithmetic progressions and separated copies}
\label{sec:progressions}

We next construct long arithmetic progressions in suitable multinomial coefficient sets and use the product embedding to place many widely separated copies of them in higher orders. The same construction will serve two different purposes: it forces omissions from every Sidon subset, and it produces disjoint contributions to pair-sum representation profiles.

For an integer $q\geq2$, define
\begin{equation}\label{eq:Nq-definition}
N(q)
=
\sum_{\substack{p\leq q\\ p\ \mathrm{prime}}}
p\left\lfloor\frac{\log q}{\log p}\right\rfloor.
\end{equation}

\begin{proposition}[Arithmetic-progression construction]
\label{prop:arithmetic-progression}
For every integer $q\geq2$, the set $\mathcal{M}_{N(q)}$ contains an arithmetic progression
\[
\lambda,2\lambda,\ldots,q\lambda
\]
for some positive integer $\lambda$.
\end{proposition}

\begin{proof}
For every prime $p\leq q$, put
\[
e_p
=
\left\lfloor\frac{\log q}{\log p}\right\rfloor
=
\max\{e\geq1:p^e\leq q\}.
\]
Consider the partition of $N(q)$ containing exactly $e_p$ parts equal to $p$ for every prime $p\leq q$. Define
\[
\Delta_q=\prod_{p\leq q}(p!)^{e_p},
\qquad
\lambda=\frac{N(q)!}{\Delta_q}.
\]

Let $1\leq j\leq q$, and write
\[
j=\prod_p p^{f_p}.
\]
Since $p^{f_p}\leq j\leq q$, we have $f_p\leq e_p$. For every prime divisor $p$ of $j$, replace $f_p$ of the parts equal to $p$ by the two parts $p-1$ and $1$. The sum of the parts remains unchanged, while the product of their factorials changes from $\Delta_q$ to
\[
\frac{\Delta_q}{\prod_p p^{f_p}}
=
\frac{\Delta_q}{j}.
\]
The resulting multinomial coefficient is therefore $j\lambda$. Hence
\[
\lambda,2\lambda,\ldots,q\lambda
\in\mathcal{M}_{N(q)}.
\]
\end{proof}

The size of the order $N(q)$ is asymptotically determined by the first powers of the primes.

\begin{lemma}
\label{lem:Nq-asymptotic}
As $q\to\infty$,
\[
N(q)\sim\frac{q^2}{2\log q}.
\]
\end{lemma}

\begin{proof}
Since a prime $p$ contributes once for each integer $j\geq1$ satisfying $p^j\leq q$, we may write
\[
N(q)
=
\sum_{j\geq1}\ \sum_{p\leq q^{1/j}}p.
\]
By the prime number theorem and partial summation,
\[
\sum_{p\leq x}p
\sim
\frac{x^2}{2\log x}.
\]
Thus the term $j=1$ contributes
\[
\frac{q^2}{2\log q}(1+o(1)).
\]
For $j\geq2$, only primes $p\leq\sqrt q$ occur, and there are at most $\log_2q$ nonempty outer sums. Consequently,
\[
\sum_{j\geq2}\ \sum_{p\leq q^{1/j}}p
\leq
(\log_2q)\sum_{p\leq\sqrt q}p
=O(q),
\]
which is negligible compared with $q^2/\log q$.
\end{proof}

We now place many separated copies of the progression from Proposition~\ref{prop:arithmetic-progression} inside a single higher-order set.

\begin{proposition}[Separated progression copies]
\label{prop:separated-progressions}
Let $q\geq2$ and
\[
n\geq N(q)+q.
\]
Put
\[
r=N(q),\qquad m=n-r,
\qquad
h=m-q+1=n-N(q)-q+1.
\]
Then $\mathcal{M}_n$ contains $h$ arithmetic progressions
\[
P_j=a_j\{1, 2,\ldots,q\},
\qquad
0\leq j\leq h-1,
\]
such that
\[
a_{j+1}>qa_j
\qquad
(0\leq j<h-1).
\]
Consequently, the intervals
\[
[a_j,qa_j],
\qquad
0\leq j\leq h-1,
\]
are pairwise disjoint, and the pair-sum intervals
\[
[2a_j,2qa_j],
\qquad
0\leq j\leq h-1,
\]
are also pairwise disjoint.
\end{proposition}

\begin{proof}
Let
\[
\lambda\{1,\ldots,q\}
\subseteq
\mathcal{M}_r
\]
be the progression supplied by Proposition~\ref{prop:arithmetic-progression}. For
\[
0\leq j\leq m-q,
\]
put
\[
y_j=\frac{m!}{(m-j)!}.
\]
The value $y_j$ belongs to $\mathcal{M}_m$, since it is represented by the partition
\[
(m-j,1^j)
\]
of $m$. Moreover, for $0\leq j<m-q$,
\[
\frac{y_{j+1}}{y_j}
=
m-j
\geq q+1
>
q.
\]
By Proposition~\ref{prop:product-embedding},
\[
P_j
=
\binom nr\lambda y_j\{1,\ldots,q\}
\subseteq
\mathcal{M}_n.
\]
Setting
\[
a_j=\binom nr\lambda y_j
\]
gives
\[
\frac{a_{j+1}}{a_j}
=
\frac{y_{j+1}}{y_j}
>
q.
\]
Hence $qa_j<a_{j+1}$, so the intervals $[a_j,qa_j]$ are pairwise disjoint. Multiplying this inequality by $2$ shows that
\[
2qa_j<2a_{j+1},
\]
and therefore the intervals $[2a_j,2qa_j]$ are pairwise disjoint as well.
\end{proof}

The first disjointness statement allows restrictions on a Sidon subset to be applied independently inside each progression $P_j$. The second ensures that the internal pair-sum representation profiles of different progression copies occur on disjoint sum values and may therefore be added without overlap.
\section{The Sidon defect}\label{sec:defect}
Let $F(q)$ denote the maximum size of a Sidon subset of $\{1,\ldots,q\}$. It is classical that
\[
F(q)\sim\sqrt q;
\]
see~\cite{OBryant} and the references therein.

\begin{proposition}\label{prop:defect-finite}
For every integer $q\geq2$ and every $n\geq N(q)+q$,
\[
M(n)-s(n)\geq(n-N(q)-q+1)(q-F(q)).
\]
\end{proposition}

\begin{proof}
By Proposition~\ref{prop:separated-progressions}, the set $\mathcal{M}_n$ contains
\[
h=n-N(q)-q+1
\]
pairwise disjoint dilates $P_j=a_j\{1,\ldots,q\}$. A Sidon subset meets each $P_j$ in at most $F(q)$ elements, because dilation preserves pair-sum equalities. Thus at least $q-F(q)$ elements must be omitted from each of the $h$ progressions, which proves the claim.
\end{proof}

\begin{proof}[Proof of Theorem~\ref{thm:defect}]
Choose an integer $q=q(n)$ satisfying
\[
q\sim c\sqrt{n\log n},\qquad 0<c<1.
\]
Then $\log q\sim\tfrac12\log n$, and Lemma~\ref{lem:Nq-asymptotic} gives
\[
N(q)\sim\frac{c^2n\log n}{2\log q}\sim c^2n.
\]
In particular, $N(q)+q<n$ for all sufficiently large $n$. Moreover, $F(q)=o(q)$. Proposition~\ref{prop:defect-finite} gives
\[
M(n)-s(n)\geq\bigl((1-c^2)+o(1)\bigr)n\bigl(c+o(1)\bigr)\sqrt{n\log n}.
\]
Hence
\[
\liminf_{n\to\infty}\frac{M(n)-s(n)}{n^{3/2}\sqrt{\log n}}\geq c(1-c^2).
\]
The maximum over $0<c<1$ is attained at $c=1/\sqrt3$, with value $2/(3\sqrt3)$.
\end{proof}

\section{Pair-sum representation functions and collision statistics}\label{sec:collision}
Let $A$ be a finite set of positive integers. For an integer $t$, define
\[
r_A(t)=\#\{(x,y):x,y\in A,\ x\leq y,\ x+y=t\}.
\]
Set
\begin{align*}
C(A)&=\sum_t\max(r_A(t)-1, 0),\\
D(A)&=\#\{t:r_A(t)\geq2\},\\
\mu(A)&=\max_t r_A(t),\\
T(A,k)&=\#\{t:r_A(t)\geq k\}.
\end{align*}
Thus $C(n)=C(\mathcal{M}_n)$, and similarly for $D$, $\mu$, and $T$.
\begin{proposition}\label{prop:basic-relations}
For every finite set $A$,
\[
C(A)=\binom{|A|+1}{2}-|A+A|.
\]
Moreover,
\[
T(A,2)=D(A),
\qquad
C(A)=\sum_{k=2}^{\mu(A)}T(A,k),
\]
and
\[
\mu(A)=\max\{k:T(A,k)>0\}.
\]
The sequence $k\mapsto T(A,k)$ is weakly decreasing.
\end{proposition}
\begin{proof}
There are $\binom{|A|+1}{2}$ unordered pairs with repetition. A sum value with $r$ representations contributes $r-1$ to the difference between the number of pairs and the number of distinct sums, proving the first identity. It also contributes once to each of $T(A,2),\ldots,T(A,r)$, proving the second identity. The remaining assertions follow directly from the definitions.
\end{proof}
\begin{proposition}\label{prop:monotone-set}
If $A\subseteq B$, then
\[
C(A)\leq C(B),\qquad D(A)\leq D(B),\qquad
\mu(A)\leq\mu(B),\qquad T(A,k)\leq T(B,k).
\]
\end{proposition}
\begin{proof}
Every representation by elements of $A$ remains a representation by elements of $B$. Thus $r_B(t)\geq r_A(t)$ for every $t$, which implies all four inequalities.
\end{proof}
By Corollary~\ref{cor:successive-embedding},
\[
(n+1)\mathcal{M}_n\subseteq\mathcal{M}_{n+1}.
\]
All four statistics are invariant under dilation by a positive integer. Proposition~\ref{prop:monotone-set} therefore shows that $C(n)$, $D(n)$, $\mu(n)$, and $T(n,k)$ are weakly increasing in $n$ for every fixed $k$.

\section{The exact pair-sum profile of an arithmetic progression}\label{sec:profile}
Let $I_q=\{1, 2,\ldots,q\}$. Scaling does not change representation multiplicities.
\begin{lemma}\label{lem:AP-profile}
For $2\leq t\leq2q$, let
\[
\rho_q(t)=\#\{(i,j):1\leq i\leq j\leq q,\ i+j=t\}.
\]
Then
\[
\rho_q(t)=
\begin{cases}
\lfloor t/2\rfloor,&2\leq t\leq q+1,\\
\lfloor t/2\rfloor-t+q+1,&q+1<t\leq2q.
\end{cases}
\]
For $2\leq k\leq\lceil q/2\rceil$, the values of $t$ satisfying $\rho_q(t)\geq k$ are precisely
\[
2k,2k+1,\ldots,2q-2k+2.
\]
Thus
\[
T(I_q,k)=2q-4k+3.
\]
\end{lemma}
\begin{proof}
For a fixed $t$, the first coordinate satisfies
\[
\max(1,t-q)\leq i\leq\left\lfloor\frac t2\right\rfloor.
\]
Counting these integers gives the formula for $\rho_q$. The inequality $\rho_q(t)\geq k$ is equivalent to $t\geq2k$ on the increasing half and to $t\leq2q-2k+2$ on the decreasing half. The resulting interval contains $2q-4k+3$ integers.
\end{proof}
\begin{corollary}\label{cor:AP-statistics}
For $q\geq3$,
\[
C(I_q)=\frac{(q-1)(q-2)}2,
\qquad D(I_q)=2q-5,
\qquad \mu(I_q)=\left\lceil\frac q2\right\rceil.
\]
\end{corollary}
\begin{proof}
There are $\binom{q+1}{2}$ unordered pairs and $2q-1$ distinct sums, so
\[
C(I_q)=\binom{q+1}{2}-(2q-1)=\frac{(q-1)(q-2)}2.
\]
The identity $D(I_q)=T(I_q,2)=2q-5$ follows from Lemma~\ref{lem:AP-profile}. The maximum of $\rho_q$ equals $\lceil q/2\rceil$.
\end{proof}
We also use the ordered additive energy
\[
E^+(A)=\#\{(x_1,x_2,x_3,x_4)\in A^4:x_1+x_2=x_3+x_4\}
=\sum_t R_A(t)^2,
\]
where $R_A(t)$ counts ordered pairs $(x,y)\in A^2$ with $x+y=t$; see~\cite{TaoVu}.
\begin{lemma}\label{lem:AP-energy}
For every $q\geq1$,
\[
E^+(I_q)=\frac{2q^3+q}{3}.
\]
\end{lemma}
\begin{proof}
The ordered representation counts are
\[
1, 2,\ldots,q-1,q,q-1,\ldots,2, 1.
\]
Therefore
\[
E^+(I_q)=q^2+2\sum_{j=1}^{q-1}j^2=\frac{2q^3+q}{3}.
\]
\end{proof}

\section{Many separated multiplicity profiles}\label{sec:manyprofiles}
\begin{proof}[Proof of Theorem~\ref{thm:finite-main}]
Let $P_0,\ldots,P_{h-1}$ be supplied by Proposition~\ref{prop:separated-progressions}. Their internal pair-sum intervals are pairwise disjoint, and each $P_j$ is a dilation of $I_q$.

For every internal sum of $P_j$, all representations inside $P_j$ remain representations in $\mathcal{M}_n$. Other elements of $\mathcal{M}_n$ can only increase its multiplicity. Since different copies concern distinct sum values, their contributions may be added. Lemma~\ref{lem:AP-profile} and Corollary~\ref{cor:AP-statistics} give
\[
C(n)\geq h\frac{(q-1)(q-2)}2,
\qquad
D(n)\geq h(2q-5),
\]
and
\[
T(n,k)\geq h(2q-4k+3)
\]
for $2\leq k\leq\lceil q/2\rceil$. One copy already contains a sum with $\lceil q/2\rceil$ representations, proving the assertion for $\mu(n)$.
\end{proof}
\begin{proposition}\label{prop:finite-energy}
For $q\geq2$ and $n\geq N(q)+q$,
\[
E^+(\mathcal{M}_n)
\geq\bigl(n-N(q)-q+1\bigr)\frac{2q^3+q}{3}.
\]
\end{proposition}
\begin{proof}
For each progression $P_j$, sum the squared ordered representation counts over its internal sum interval. These intervals are pairwise disjoint, so no sum value receives contributions from two different progression copies. The full representation count in $\mathcal{M}_n$ is at least the internal count of $P_j$. Lemma~\ref{lem:AP-energy} gives the result.
\end{proof}

\section{Asymptotic multiplicity consequences}\label{sec:asymmult}
For fixed $0<c<1$, choose the integer
\[
q=\left\lfloor c\sqrt{n\log n}\right\rfloor.
\]
Then $q\sim c\sqrt{n\log n}$ and $q=o(n)$. Since $\log q\sim\tfrac12\log n$, Lemma~\ref{lem:Nq-asymptotic} gives
\[
N(q)\sim\frac{q^2}{2\log q}=\bigl(c^2+o(1)\bigr)n,
\qquad
n-N(q)-q+1=\bigl(1-c^2+o(1)\bigr)n.
\]
\begin{proof}[Proof of Theorem~\ref{thm:asymptotic-main}]
Theorem~\ref{thm:finite-main} gives
\[
C(n)\geq\left(\frac12c^2(1-c^2)+o(1)\right)n^2\log n.
\]
The coefficient is maximized at $c=1/\sqrt2$, where it equals $1/8$.

Similarly,
\[
D(n)\geq\left(2c(1-c^2)+o(1)\right)n^{3/2}\sqrt{\log n}.
\]
The coefficient is maximized at $c=1/\sqrt3$, where it equals $4/(3\sqrt3)$.

For the maximum multiplicity, any fixed $0<c<1$ gives
\[
\mu(n)\geq\left(\frac c2+o(1)\right)\sqrt{n\log n}.
\]
Taking the supremum over fixed $c<1$ proves the last assertion.
\end{proof}
\begin{proof}[Proof of Theorem~\ref{thm:profile-limit}]
Fix $c$ with $2\alpha<c<1$, and choose $q=\lfloor c\sqrt{n\log n}\rfloor$. The hypothesis on $k_n$ implies $k_n\leq\lceil q/2\rceil$ for all sufficiently large $n$. Theorem~\ref{thm:finite-main} yields
\[
T(n,k_n)\geq
\bigl((1-c^2)+o(1)\bigr)n
\bigl((2c-4\alpha)+o(1)\bigr)\sqrt{n\log n}.
\]
Hence
\[
\liminf_{n\to\infty}\frac{T(n,k_n)}{n^{3/2}\sqrt{\log n}}
\geq2(c-2\alpha)(1-c^2).
\]
The derivative of $(c-2\alpha)(1-c^2)$ vanishes when
\[
3c^2-4\alpha c-1=0.
\]
For $0\leq\alpha<1/2$, its unique root in $(2\alpha,1)$ is
\[
c=c_\alpha=\frac{2\alpha+\sqrt{4\alpha^2+3}}3.
\]
Substitution gives $\Phi(\alpha)$.
\end{proof}
\begin{corollary}\label{cor:energy-asymp}
We have
\[
\liminf_{n\to\infty}
\frac{E^+(\mathcal{M}_n)}{n^{5/2}(\log n)^{3/2}}
\geq\frac{4\sqrt{15}}{125}.
\]
\end{corollary}
\begin{proof}
Proposition~\ref{prop:finite-energy} gives
\[
E^+(\mathcal{M}_n)\geq
\left(\frac23c^3(1-c^2)+o(1)\right)n^{5/2}(\log n)^{3/2}.
\]
The coefficient is maximized at $c=\sqrt{3/5}$, where it equals $4\sqrt{15}/125$.
\end{proof}

\section{Dependence on the number of variables}\label{sec:variables}
For $1\leq v\leq n$, let
\[
\mathcal{M}_{n,v}=
\left\{
\frac{n!}{e_1!\cdots e_v!}:
 e_1+\cdots+e_v=n,\quad e_i\geq0
\right\}
\]
with repeated values removed. Equivalently, $\mathcal{M}_{n,v}$ consists of the values corresponding to partitions of $n$ into at most $v$ positive parts. Thus
\[
\mathcal{M}_{n,1}\subseteq\mathcal{M}_{n,2}\subseteq\cdots\subseteq\mathcal{M}_{n,n}=\mathcal{M}_n.
\]
Define $C(n,v)$, $D(n,v)$, $\mu(n,v)$, and $T(n,v,k)$ by replacing $\mathcal{M}_n$ with $\mathcal{M}_{n,v}$.
\begin{proposition}\label{prop:variables-monotone}
For fixed $n$, all four statistics are weakly increasing in $v$.
\end{proposition}
\begin{proof}
This follows from set inclusion and Proposition~\ref{prop:monotone-set}.
\end{proof}
\begin{lemma}\label{lem:max-by-parts}
Let $(p_1,\ldots,p_r)$ be a partition of $n$ into $r$ positive parts. Then
\[
\prod_{i=1}^r p_i!\geq2^{n-r},
\]
with equality if and only if every part is $1$ or $2$. Consequently, if $d\geq0$ and $n\geq2d$, then
\[
\max\mathcal{M}_{n,n-d}=\frac{n!}{2^d},
\]
achieved by the partition $(2^d,1^{n-2d})$.
\end{lemma}
\begin{proof}
For every integer $a\geq1$,
\[
a!\geq2^{a-1},
\]
with equality exactly for $a=1, 2$. Multiplication gives
\[
\prod_{i=1}^r p_i!\geq2^{\sum_i(p_i-1)}=2^{n-r}.
\]
If a partition has at most $n-d$ parts, then $n-r\geq d$, so its factorial denominator is at least $2^d$. When $n\geq2d$, the displayed partition has exactly $n-d$ parts and denominator $2^d$.
\end{proof}
In particular,
\[
\max\mathcal{M}_{n,n-1}=\frac{n!}{2},\qquad
\max\mathcal{M}_{n,n-2}=\frac{n!}{4},\qquad
\max\mathcal{M}_{n,n-3}=\frac{n!}{8}
\]
under the respective hypotheses $n\geq2$, $n\geq4$, and $n\geq6$.
\begin{proposition}\label{prop:variables-stable}
For every $n\geq3$ and every $k\geq2$,
\begin{align*}
C(n,n-2)&=C(n,n-1)=C(n,n),\\
D(n,n-2)&=D(n,n-1)=D(n,n),\\
\mu(n,n-2)&=\mu(n,n-1)=\mu(n,n),\\
T(n,n-2,k)&=T(n,n-1,k)=T(n,n,k).
\end{align*}
For $n\geq5$,
\[
C(n,n-3)<C(n,n-2),
\qquad
D(n,n-3)<D(n,n-2).
\]
\end{proposition}
\begin{proof}
The case $n=3$ is immediate. Let $n\geq4$. Passing from $\mathcal{M}_{n,n-2}$ to $\mathcal{M}_{n,n-1}$ adds exactly $n!/2$, corresponding to $(2, 1^{n-2})$. By Lemma~\ref{lem:max-by-parts}, the largest old value is $n!/4$. Since both summands in an old pair are at most $n!/4$, every old pair sum is at most $n!/2$. Every sum involving the new value is at least $n!/2+1$, and the new sums are mutually distinct. No multiplicity at least $2$ is created or altered.

Passing from $\mathcal{M}_{n,n-1}$ to $\mathcal{M}_{n,n}$ adds exactly $n!$, corresponding to $(1^n)$. The largest previous value is $n!/2$, so, by the same two-summand bound, every old pair sum is at most $n!$. Every sum involving the new value is at least $n!+1$, and the new sums are mutually distinct. This proves all four stabilization identities.

For $n\geq6$, the values $n!/24$ and $n!/8$ already belong to $\mathcal{M}_{n,n-3}$, while $n!/6$ and $n!/4$ first occur in $\mathcal{M}_{n,n-2}$. Lemma~\ref{lem:max-by-parts} gives $\max\mathcal{M}_{n,n-3}=n!/8$, so every old pair sum is at most $n!/4$. The equality
\[
\frac{n!}{4}+\frac{n!}{24}
=
\frac{n!}{6}+\frac{n!}{8}
=\frac{7n!}{24}
\]
is therefore a genuinely new colliding sum, proving strict increase of both $C$ and $D$. For $n=5$, use $10+30=20+20$.
\end{proof}
\begin{remark}
The stabilization is specific to the last two adjunctions. It does not assert stability in earlier columns, and the two-parameter growth of $C(n,v)$ and $D(n,v)$ remains open.
\end{remark}

\section{An explicit family of trinomial collisions}\label{sec:trinomial}
Let $\mathcal{T}_n=\mathcal{M}_{n,3}$ and write $C_3(n)=C(\mathcal{T}_n)$ and $D_3(n)=D(\mathcal{T}_n)$.

For $m\geq3$ and $2\leq q\leq m$, define
\begin{align*}
A_q&=\binom{2m+1}{m,\,m-q+1,\,q},\\
B_q&=\binom{2m+1}{m+1,\,m-q+2,\,q-2},\\
U_q&=\binom{2m+1}{m,\,m-q+2,\,q-1},\\
V_q&=\binom{2m+1}{m+1,\,m-q,\,q}.
\end{align*}
\begin{lemma}\label{lem:trinomial-identity}
For every $m\geq3$ and $2\leq q\leq m$,
\[
A_q+B_q=U_q+V_q.
\]
\end{lemma}
\begin{proof}
Divide all four terms by
\[
K_q=\frac{(2m+1)!}{(m+1)!q!(m-q+2)!}.
\]
The normalized terms are
\[
A_q=(m+1)(m-q+2)K_q,\quad
B_q=q(q-1)K_q,
\]
\[
U_q=(m+1)qK_q,\quad
V_q=(m-q+2)(m-q+1)K_q.
\]
The required equality reduces to
\[
(m+1)(m-q+2)+q(q-1)
=(m+1)q+(m-q+2)(m-q+1),
\]
which follows by expansion.
\end{proof}
Let $S_q=A_q+B_q$. Direct simplification gives
\[
S_q=
\frac{\bigl((m+1)(m+2)-q(m+2-q)\bigr)(2m+1)!}
{q!(m+1)!(m+2-q)!},
\]
so $S_q=S_{m+2-q}$.
\begin{lemma}\label{lem:trinomial-distinct}
For distinct $q,r\in\{2,\ldots,m\}$,
\[
S_q=S_r\quad\Longleftrightarrow\quad q+r=m+2.
\]
Moreover, Lemma~\ref{lem:trinomial-identity} gives a genuine collision except when $m$ is even and $q=(m+2)/2$. The symmetric parameters $q$ and $m+2-q$ describe the same equality with its sides reversed.
\end{lemma}
\begin{proof}
A direct calculation shows that the sign of $S_{q+1}-S_q$ is the sign of
\[
(m+1-2q)\bigl(q^2-qm-q+m^2+2m\bigr).
\]
The second factor, viewed as a quadratic in $q$, has minimum
\[
\frac{3m^2+6m-1}{4}>0.
\]
Thus $S_q$ is strictly increasing up to the center and strictly decreasing after the center. Together with $S_q=S_{m+2-q}$, this proves the first assertion.

Put $t=m-q+2$. After division by $K_q$, the two unordered pairs are
\[
\{(m+1)t,\ q(q-1)\}
\qquad\text{and}\qquad
\{(m+1)q,\ t(t-1)\}.
\]
If corresponding entries agree, then $t=q$. Under crossed matching, division by the positive integers $q,t$ would give both $t=m+2$ and $q=m+2$, impossible for $2\leq q,t\leq m$. Hence the pairs coincide only when $t=q$, that is, when $m$ is even and $q=(m+2)/2$.

Finally, for $q'=m+2-q$, comparison of factorial denominators gives
\[
A_{q'}=U_q,\quad B_{q'}=V_q,\quad C_{q'}=A_q,\quad D_{q'}=B_q,
\]
which proves the last statement.
\end{proof}
\begin{theorem}\label{thm:trinomial-linear}
For every odd $n=2m+1\geq7$,
\[
D_3(n)\geq\left\lfloor\frac{m-1}{2}\right\rfloor
=\left\lfloor\frac{n-3}{4}\right\rfloor.
\]
Consequently,
\[
C_3(n)\geq\left\lfloor\frac{n-3}{4}\right\rfloor.
\]
\end{theorem}
\begin{proof}
The involution $q\mapsto m+2-q$ acts on $\{2,\ldots,m\}$. By Lemmas~\ref{lem:trinomial-identity} and~\ref{lem:trinomial-distinct}, its nontrivial orbits give distinct genuine colliding sums. There are $\lfloor(m-1)/2\rfloor$ such orbits. Each contributes at least one to the collision excess.
\end{proof}
Every additive identity among binomial coefficients in row $r$ lifts to a trinomial identity of every order $n\geq r$, since
\[
\binom{n}{n-r,r-k,k}=\binom nr\binom rk.
\]
Thus the trinomial problem receives both the intrinsic family above and lifted families from Pascal's triangle.

\section{Exact computations, OEIS connections, and open problems}\label{sec:computations}

\subsection{Exact Sidon-subset computations}
For a finite set $A$, form a hypergraph $\mathcal{H}(A)$ whose vertex set is $A$, and whose hyperedges are the supports of nontrivial equalities
\[
x+y=u+v
\]
between distinct unordered pairs from $A$. A subset of $A$ is Sidon exactly when it is independent in $\mathcal{H}(A)$. Consequently, $M(n)-s(n)$ is the transversal number of $\mathcal{H}(\mathcal{M}_n)$.

The distinct values in $\mathcal{M}_n$ were generated, all inclusion-minimal collision supports were formed, and the resulting minimum-transversal problem was solved by a SAT encoding with a cardinality constraint. A vertex is called active if it belongs to at least one minimal forbidden support. The computation was completed in Maple 2026.1 (Build ID 2018217) on Windows 11 Enterprise, using the \texttt{maplesat} backend. For $n=16$, the reduced instance has $|\mathcal{M}_{16}|=157$, 146 active vertices, and 2581 minimal forbidden supports. The instance with 74 deletions is satisfiable, whereas the instance with 73 deletions is unsatisfiable. Hence
\[
s(16)=157-74=83.
\]
The exact values are
\begin{center}
\small
\begin{tabular}{c|rrrrrrrrrrrrrrrr}
$n$&1&2&3&4&5&6&7&8&9&10&11&12&13&14&15&16\\
\hline
$s(n)$&1&2&3&5&6&8&11&15&20&25&30&39&45&55&66&83
\end{tabular}
\end{center}

For $n=5$,
\[
\mathcal{M}_5=\{1, 5, 10, 20, 30, 60, 120\},
\qquad
10+30=20+20.
\]
The set
\[
\{1, 5, 20, 30, 60, 120\}
\]
is Sidon, which gives the elementary verification $s(5)=6$.

A separate search found a Sidon subset of $\mathcal{M}_{17}$ with 89 elements, so $s(17)\geq89$; no optimality claim is made.

\subsection{Pair-sum data}
The first values of the principal collision statistics are as follows.
\begin{center}
\begin{tabular}{rrrrr}
\toprule
$n$&$M(n)$&$C(n)$&$D(n)$&$\mu(n)$\\
\midrule
1&1&0&0&1\\
2&2&0&0&1\\
3&3&0&0&1\\
4&5&0&0&1\\
5&7&1&1&2\\
6&11&6&6&2\\
7&14&8&8&2\\
8&20&24&24&2\\
9&27&43&42&3\\
10&36&74&68&3\\
11&47&148&128&4\\
12&64&245&198&4\\
13&79&441&336&5\\
14&102&753&557&6\\
15&125&1178&781&6\\
16&157&1742&1119&7\\
17&193&2605&1623&8\\
18&243&3805&2311&9\\
19&296&5589&3321&9\\
20&366&8256&4809&10\\
\bottomrule
\end{tabular}
\end{center}
For example,
\[
\mathcal{M}_6=\{1, 6, 15, 20, 30, 60, 90, 120, 180, 360, 720\}.
\]
There are $\binom{12}{2}=66$ unordered pairs with repetition but only 60 distinct sums, so $C(6)=6$. At $n=9$, the sum 1764 has three representations,
\[
84+1680=252+1512=504+1260,
\]
which explains the first strict inequality $C(n)>D(n)$.

The sequences $s(n)$, $C(n)$, $D(n)$, and $\mu(n)$ are OEIS A398429, A398233, A398238, and A398263, respectively. The cumulative profile is recorded in A398265. The variable-number triangles $C(n,v)$ and $D(n,v)$ are A398232 and A398237, and the trinomial statistics $C_3(n)$ and $D_3(n)$ are A398230 and A398231~\cite{OEIS}.

\subsection{Scope of the computational checks}
The ancillary Maple program for $s(n)$ constructs the collision hypergraph and performs the SAT computations. A second Maple program generates $\mathcal{M}_{n,v}$ and the collision statistics, reproduces the principal table through $n=20$, checks the last-three-column stabilization and the strict preceding increase through $n=20$, verifies the trinomial family through the half-order parameter $m=80$, and verifies the arithmetic-progression construction through $q=10$. An independent standard-library Python implementation performs the same pair-sum checks. All these computations use exact integer arithmetic. The asymptotic theorems and the symbolic proofs of the finite constructions do not depend on the finite audits. No solver-independent certificate for the SAT optimality computations is presently included.

\subsection{Open problems concerning Sidon subsets}
\begin{question}
What is the order of magnitude of $s(n)$? Is the logarithmic-exponential scale in Theorem~\ref{thm:main-lower} optimal?
\end{question}
\begin{question}
Does there exist an absolute constant $K>0$ such that $s(n)=O(n^K)$?
\end{question}
\begin{question}
Determine $s(17)$ and compute further exact values.
\end{question}
\begin{question}
Determine the order of magnitude of $M(n)-s(n)$.
\end{question}
\begin{question}
Is $b(n)=(1-o(1))s(n)$? More generally, determine the asymptotic behavior of $b(n)/s(n)$, which lies between $1/3$ and $1$ by Corollary~\ref{cor:b-vs-s}.
\end{question}

\subsection{Open problems concerning pair-sum multiplicities}
\begin{question}
Determine the true orders of magnitude of $C(n)$, $D(n)$, and $\mu(n)$. Are the lower bounds in Theorem~\ref{thm:asymptotic-main} of the correct scale?
\end{question}
\begin{question}
Does either ratio
\[
\frac{C(n)}{n^2\log n},\qquad
\frac{D(n)}{n^{3/2}\sqrt{\log n}}
\]
have a limit, or even a finite positive limsup?
\end{question}
\begin{question}
Is there an upper limit shape for the normalized profile $k\mapsto T(n,k)$? How close is $\Phi$ to the true lower envelope?
\end{question}
\begin{question}
Determine the order of magnitude of $E^+(\mathcal{M}_n)$.
\end{question}
\begin{question}
For fixed $v\geq3$, determine the growth of the $v$-nomial collision statistics $C(n,v)$ and $D(n,v)$.
\end{question}
\begin{question}
Classify primitive pair-sum identities among multinomial coefficients, where primitive means that the identity is not obtained from a lower order by the product embedding or by appending parts equal to $1$.
\end{question}
\section{AI-assisted research disclosure}\label{sec:ai}
The mathematical exploration and preparation of this manuscript were carried out with substantial assistance from OpenAI's ChatGPT, principally GPT-5.6 Thinking, during July and August 2026. The system was used interactively to explore conjectures and generalizations, propose definitions and intermediate claims, develop and revise proof arguments, identify possible gaps and notation conflicts, assist with symbolic calculations, produce and review Maple and Python code, design computational audits, and draft and edit the manuscript.

The author selected the research questions and final statements, supplied and interpreted the motivating OEIS data, directed the successive investigations, evaluated proposed approaches, executed the stated Maple audits, reviewed the Python verification and its output, and checked the complete manuscript and computational package. The author reviewed every theorem and proof presented here and accepts responsibility for the submitted work. Some advanced arguments were nevertheless developed with substantial AI assistance and were not independently rediscovered or reconstructed by the author from first principles without that assistance. This limitation is disclosed explicitly so that readers and referees can assess the provenance and verification status of the work. Any error remains the author's responsibility, and the author undertakes to investigate and correct issues that may be identified.

\section*{Code and data availability}
The ancillary archive contains the Maple sources for the Sidon and pair-sum computations, the completed SAT console output for $n=14, 15, 16$, a computational supplement describing the SAT encoding and environment, an independent Python verification, the recorded successful audit outputs, and a README describing reproduction and scope. The exact filenames are listed in the README. No external dataset is required.

\section*{Acknowledgments}
The author thanks the OEIS community for providing the setting in which the motivating sequences were developed.

\end{document}